\documentclass{amsart}
\usepackage{amsmath,amsthm,amssymb}
\usepackage{hyperref}
\hypersetup{
    colorlinks,
    linkcolor = blue!50!black,
    citecolor = orange
}
\usepackage{quiver}
\usepackage{mathtools}
\usepackage{stmaryrd}
\usepackage{biblatex}
\usepackage{cleveref}
\usepackage{enumitem}

\usepackage{todonotes}

\DeclareMathOperator{\Sub}{\mathsf{Sub}}
\DeclareMathOperator{\Sat}{\mathsf{Sat}}
\DeclareMathOperator{\Tr}{\mathsf{Tr}}
\DeclareMathOperator{\coTr}{\mathsf{coTr}}
\DeclareMathOperator{\Cov}{Cov}

\DeclarePairedDelimiter{\floor}{\lfloor}{\rfloor}
\DeclarePairedDelimiter{\ceil}{\lceil}{\rceil}

\newcommand{\qbinom}[3]{%
  \IfNoValueTF{#3}{\binom{#1}{#2}}{%
    \mathchoice
    {\binom{#1}{#2}_{\!\!#3}}
    {\binom{#1}{#2}_{\!#3}}
    {\binom{#1}{#2}_{\!#3}}
    {\binom{#1}{#2}_{\!#3}}%
  }%
}

\newtheorem{lettertheorem}{Theorem}

\newtheorem{theorem}{Theorem}[section]
\newtheorem{proposition}[theorem]{Proposition}
\newtheorem{lemma}[theorem]{Lemma}
\newtheorem{corollary}[theorem]{Corollary}
\theoremstyle{definition}
\newtheorem{definition}[theorem]{Definition}

\title{The Formal Context of Saturated Transfer Systems on Finite Abelian Groups}
\author[S.\ Bernstein]{Seth Bernstein}
\address{Department of Mathematics, University of Virginia, Charlottesville, VA, USA}
\email{wsw9sq@virginia.edu}
\author[B.\ Spitz]{Ben Spitz}
\address{Department of Mathematics, Indiana University, Bloomington, IN, USA}
\email{bespitz@iu.edu}

\begin{document}

\begin{abstract}
    We describe the reduced formal context of the lattice of saturated transfer systems on a finite abelian group. As an application, we compute that there are $13,784,538,270,571$ saturated transfer systems on the elementary abelian group $C_5^3$.
\end{abstract}

\maketitle

\tableofcontents

\section{Introduction}

In non-equivariant (i.e.\ ordinary) homotopy theory, it is a central result that all $E_\infty$-operads are weakly equivalent, i.e.\ there is a single essentially unique way to specify a homotopy-coherent multiplication operation on a space or spectrum. Work of Blumberg and Hill \cite{blumberghill} shows that this is no longer the case in equivariant homotopy theory -- the \emph{$N_\infty$-operads} they introduce parametrize ways to specify homotopy-coherent multiplications on $G$-spaces and $G$-spectra, and they show that (when $G$ is not the trivial group) it is \emph{not} the case that all $N_\infty$-operads are weakly equivalent. However, Blumberg--Hill \cite{blumberghill} together with the work of many others \cite{bbr, BP_operads, GW_operads, Rubin_comb} establishes that (weak-equivalence classes of) $N_\infty$-operads are in natural bijective correspondence with \emph{transfer systems}, which are simple combinatorial objects. In particular, fixing a finite group $G$, there is a finite lattice $\Tr(\Sub(G))$ (the \emph{lattice of $G$-transfer systems}) which is equivalent to the homotopy category of $N_\infty$-operads, and thus whose structure yields information about algebras in the $G$-equivariant stable homotopy category. The field of \emph{homotoptical combinatorics} is primarily concerned with studying these lattices $\Tr(\Sub(G))$ and related structures.

A $G$-transfer system is simply a subcategory of the subgroup lattice $\Sub(G)$ satisfying some conditions. In principle, this makes the lattice $\Tr(\Sub(G))$ straightforward to compute. In practice, however, computing this lattice is very computationally intensive. To date, the largest such lattice which has been computed is $\Tr(\Sub(A_6))$ by S.\ Balchin and the second author \cite{BS}, having approximately $37.8$ billion elements, eclipsing the previous record. To enable this computation, Balchin and the second author leveraged the theory and tools of \emph{Formal Concept Analysis} (FCA).

Developed in the early 1980s, FCA makes the fundamental observation that a finite lattice $(L,{\leq})$ is determined up to isomorphism by the restriction of the relation ${\leq}$ (which is a subset of $L \times L$) to a potentially much smaller set $J(L) \times M(L)$ (to be defined later). Thus, if one wishes to compute the lattice $\Tr(\Sub(G))$, one needs only:
\begin{enumerate}
    \item Compute $J(\Tr(\Sub(G)))$ and $M(\Tr(\Sub(G)))$,
    \item Compute the desired relation between $J(\Tr(\Sub(G)))$ and $M(\Tr(\Sub(G)))$,
    \item Use existing FCA software tools to reconstruct the lattice $\Tr(\Sub(G))$.
\end{enumerate}
In \cite{BS}, Balchin and the second author show that steps 1 and 2 can be done very quickly (both abstractly and computationally), and the combined algorithm (steps 1-3) is much faster than previously-known algorithms for computing $\Tr(\Sub(G))$.

In this paper, we take the same approach to computing $\Sat(\Sub(G))$, the lattice of \emph{saturated transfer systems} on a finite group $G$, which are closely related to linear isometries operads \cite{Rubin}. We specialize to the case of finite abelian groups (and elementary abelian groups in particular), where we are able to execute steps 1-2 above.

To simplify the discussion, we often replace $\Sub(G)$ with an arbitrary finite modular lattice $L$ (note that, when $G$ is a finite abelian group, $\Sub(G)$ is a finite modular lattice). While $\Sat(L)$ is a subposet of $\Tr(L)$ with the same meet operation, it is \emph{not} always a sublattice, i.e. the join operations of $\Sat(L)$ and $\Tr(L)$ need not coincide. However, we are able to show that $J(\Sat(L))$, $M(\Sat(L))$, and the restricted relation are easy to compute in terms of $L$ alone. The theorem below uses notation following \cite{BS}, which will be introduced later and here can be taken as formal.

\begin{lettertheorem}\label{letterthm:formal context}
    Let $L$ be a finite modular lattice with minimum element $\bot$. Then
    \[J(\Sat(L)) =  J(\Tr(L)) \cap \Sat(L) = \{\floor{H \to K} : H<K \text{ is a covering relation in } L\}\]
    and
    \[M(\Sat(L)) = M(\Tr(L)) \cap \Sat(L) = \{\ceil{\bot \to X}^\boxslash : X \neq \bot\}.\]
    Moreover,
    \[\floor{H \to K} \subseteq \ceil{\bot \to X}^\boxslash \iff (X \nleq K \text{ or } X \leq H).\]
\end{lettertheorem}

As discussed above, \Cref{letterthm:formal context} allows for the complete reconstruction of the lattice $\Sat(L)$. In particular, this allows for the complete enumeration of all saturated transfer systems in the case $L = \Sub(C_5^3)$ (where $C_5^3$ denotes the elementary abelian group of order $125$).

\begin{lettertheorem}
    There are exactly $13,784,538,270,571$ saturated transfer systems on $C_5^3$.
\end{lettertheorem}

In \cite{BS} and the FCA literature \cite{albano2017polynomial}, one sometimes considers the \emph{density} $\delta(L)$ of a lattice $L$ as a measure of its complexity (valued in the interval $[0,1]$). Heuristically, keeping the sizes of $J(L)$ and $M(L)$ fixed, lattices of higher density are more difficult to reconstruct from the relation ${\leq} \subseteq J(L) \times M(L)$.

Empirically, when $\mathcal{G}$ is a natural family of finite groups (e.g.\ cyclic $p$-groups or elementary abelian groups) with limiting density
\[\lim_{n \to \infty} \sup_{\substack{G \in \mathcal{G} \\ \lvert G \rvert \leq n}} \delta(\Tr(\Sub(G)))\]
strictly less than $1$, we have somewhat clear understandings of the lattices $\Tr(\Sub(G))$. In contrast, when this limiting density is equal to $1$, the lattices $\Tr(\Sub(G))$ remain poorly understood. In \Cref{sec:density}, we compute these limiting densities for families of elementary abelian $p$-groups.

\begin{lettertheorem}\label{letterthm:density}
    With $p$ fixed, one has
    \[\lim_{n \to \infty} \delta(\Sat(\Sub(C_p^n))) = 1\]
    and with $n$ fixed, one has
    \[\lim_{p \to \infty} \delta(\Sat(\Sub(C_p^n))) = \begin{cases}
        0 &: n = 1 \\
        1/2 &: n = 2 \\
        1 &: n > 2
    \end{cases}\]
\end{lettertheorem}

These limiting densities can be connected with the above comments on empirics. When $n = 1$, $\Sat(\Sub(C_p^n))$ is the lattice with two elements, and is thus fully understood. In \cite{bao_transfer_2025}, the lattices $\Tr(\Sub(C_p^2))$ and $\Sat(\Sub(C_p^2))$ are studied, and a formula for $\lvert \Tr(\Sub(C_p^2)) \rvert$ in terms of $p$ is produced. Unfortunately, the methods of \cite{bao_transfer_2025} cannot be easily extended to the case $n > 2$, and \Cref{letterthm:density} gives some heuristic evidence that these lattices are intrinsically more complex.

\subsection*{Acknowledgments}

The authors thank Scott Balchin for his help with computations, and the use of the computing resources from the Northern Ireland High Performance Computing (NI-HPC) service funded by EPSRC (EP/T022175). We also thank the Directed Reading Program at the University of Virginia, where this project began.

\section{Background}

In this section, we introduce the relevant basic definitions and results of homotopical combinatorics.

\subsection{Primer on Saturated Transfer Systems}

\begin{definition}
    Let $(L,\leq)$ be a lattice. A \emph{transfer system} on $L$ is a partial order $\to$ on $L$ such that
    \begin{enumerate}
     \item If $x \to y$, then $x \leq y$.
     \item If $x \to y$ and $z\leq y$, then $x \wedge z \to z$.
     \end{enumerate}
     A \emph{cotransfer system} on $L$ is a partial order $\to$ on $L$ such that
     \begin{enumerate}
     \item If $x \to y$, then $x \leq y$.
     \item If $x \to y$ and $x\leq z$, then $z\to y \vee z$.
     \end{enumerate}
\end{definition}

\begin{definition}
A transfer system $\to$ is said to be \emph{saturated} if the set of arrows in $T$ satisfies the 2-out-of-3 property, i.e. for all triples $x \leq y \leq z$, if any two of $x \to y$, $x \to z$, and $y \to z$ hold, then so does the third. By the other axioms of transfer systems, this is the same as requiring that: if $x \to z$, then $y \to z$.
\end{definition}

A key feature of (saturated/co)transfer systems is that they form a lattice.

\begin{proposition}
    Let $L$ be a lattice. The set of transfer systems on $L$, ordered by containment, forms a lattice, where meet is given by intersection. We use $\Tr(L)$ to denote the lattice of transfer systems on $L$. The subposet of saturated transfer systems also forms a lattice, where meet is again given by intersection. We denote the lattice of saturated transfer systems by $\Sat(L)$.
\end{proposition}

While $\Sat(L)$ is a lattice and a subposet of $\Tr(L)$, it is not a sublattice of $\Tr(L)$ -- the join operations are not the same in general. Cotransfer systems also form a lattice, as a formal consequence of a duality implied by their name.

\begin{proposition}
    Let $L$ be a lattice. Cotransfer systems on $L$ are the same as\footnote{The correspondence between transfer systems on $L^\mathrm{op}$ and cotransfer systems on $L$ is literally given by reversing all arrows.} transfer systems on the opposite lattice $L^\mathrm{op}$, and thus the set $\coTr(L)$ of cotransfer systems on $L$ naturally forms a lattice (where meet is given by intersection).
\end{proposition}

Since the set of (co)transfer systems on $L$ is closed under intersection, one can consider for any relation $x \leq y$ in a lattice $L$ the (co)transfer system generated by this relation. Our notation for these objects follows that of \cite{BS}.

\begin{definition}
    Let $L$ be a lattice and let $x,y \in L$ with $x \leq y$. We denote by $\floor{x \to y}$ the smallest transfer system on $L$ containing the relation $x \to y$, and by $\ceil{x \to y}$ the smallest cotransfer system on $L$ containing the relation $x \to y$.

    More generally, for any set $\mathcal{S}$ of arrows in $L$, we denote by $\floor{\mathcal{S}}$ the transfer system generated by $\mathcal{S}$, by $\ceil{\mathcal{S}}$ the cotransfer system generated by $\mathcal{S}$, and by $\floor{\mathcal{S}}_\mathrm{sat}$ the saturated transfer system generated by $S$.
\end{definition}

Besides the fact that $\coTr(L) \cong \Tr(L^\mathrm{op})$, one also has an isomorphism $\coTr(L) \cong \Tr(L)^\mathrm{op}$. To describe this isomorphism, we recall the definitions of right- and left-lifting classes.

\begin{definition}
Let $L$ be a lattice and let $f : x \leq y$ and $g : x' \leq y'$ be arrows in $L$. We say that $(f,g)$ \emph{has the lifting property} if the implication ``if $x \leq x'$ and $y \leq y'$ then $y \leq x'$'' holds.
\[
\begin{tikzcd}
x \arrow[d, "f"'] \arrow[r] & x' \arrow[d,"g"] \\
y \arrow[r] \arrow[ur,dotted] & y'
\end{tikzcd}
\]
We may write $f \boxslash g$ to denote that $(f,g)$ has the lifting property.
\end{definition}
\begin{definition}
Let $X$ be a set of arrows in $L$. Then: 
    \[X^\boxslash = \{f : g\boxslash f  \text{ for all } g\in X\}\]
    \[{}^\boxslash X = \{f : f\boxslash g  \text{ for all } g\in X\}.\]
    $X^\boxslash$ is called the \emph{right lifting class} of $X$ and ${}^\boxslash X$ is called the \emph{left lifting class} of $X$.
\end{definition}

It is clear by definition that the operations ${(-)}^\boxslash$ and ${}^\boxslash{(-)}$ are order-reversing with respect to $\subseteq$.

\begin{proposition}[\cite{luo2024latticeweakfactorizationsystems}]
\label{prop:boxslash adjunction}
    Let $T$ be a partial order on a lattice $L$. If $T$ is a cotransfer system then $T^\boxslash$ is a transfer system. If T is a transfer system then ${}^\boxslash T$ is a cotransfer system. This yields an isomorphism between $\Tr(L)^\mathrm{op}$ and $\coTr(L)$:
\[\begin{tikzcd}
	{\Tr(L)^\mathrm{op}} && {\coTr(L)}
	\arrow["{{}^\boxslash({-})}"', shift right, curve={height=6pt}, from=1-1, to=1-3]
	\arrow["{({-})^\boxslash}"', shift right, curve={height=6pt}, from=1-3, to=1-1]
\end{tikzcd}\]
\end{proposition}

We now turn our attention to saturated transfer systems. In \cite{hafeez_saturated_2022} it is shown that a saturated transfer system on a finite group is determined by the covering relations it contains, and moreover the sets of covering relations which arise this way are classified. This is generalized to the context of finite modular lattices in \cite{bao_transfer_2025}. We recount this story below.

\begin{definition}
    Let $L$ be a lattice. If $x,y \in L$ are such that $x < y$ and there does not exist $z \in L$ such that $x < z < y$, then we say that $y$ \emph{covers} $x$ (or that $x < y$ is a \emph{covering relation}).
\end{definition}

\begin{definition}
    Let $L$ be a modular lattice. A \emph{saturated cover} on $L$ is a set $\mathcal{S}$ of edges in $L$ such that:
    \begin{enumerate}
        \item (Covering) Each edge in $\mathcal{S}$ is a covering relation;
        \item (Restriction) For all $x, y \in L$, if $x \xrightarrow{\mathcal{S}} x \vee y$ then $x \wedge y \xrightarrow{\mathcal{S}} y$.
        \item (3-out-of-4) If $x$ and $y$ cover $x \wedge y$, then if any three of the covering relations between $x, y, x \wedge y, x \vee y$ are in $\mathcal{S}$, so is the fourth.
    \end{enumerate}
\end{definition}

In the above definition, we make use of the fact that in a modular lattice, when $x$ and $y$ both cover $x \wedge y$, then $x \vee y$ covers both $x$ and $y$.

\begin{proposition}
    Let $L$ be a finite modular lattice. There is a bijective correspondence between saturated transfer systems on $L$ and saturated covers on $L$, given by sending a saturated cover $\mathcal{S}$ to the transfer system $\floor{\mathcal{S}}$. The inverse bijection is given by sending a saturated transfer system $\mathsf{T}$ to the set of covering relations it contains.
\end{proposition}

\subsection{Primer on Formal Concept Analysis}

Formal Concept Analysis (FCA) is built on the fundamental observation that a finite lattice $(L,\leq)$ is completely determined by the $\leq$-relations between its join-irreducible and meet-irreducible elements.

In any lattice $L$, we will use $\bot$ to denote the minimum element of $L$ and $\top$ to denote the maximum element of $L$.

\begin{definition}
    Let $L$ be a finite lattice. An element $x \in L$ is said to be \emph{join-irreducible} if
    \begin{enumerate}[label=(\roman*)]
        \item $x \neq \bot$;
        \item For all $a, b \in L$, if $x = a \vee b$, then $x = a$ or $x = b$.
    \end{enumerate}
    Dually, an element $x \in L$ is said to be \emph{meet-irreducible} if $x \neq \top$ and $x = a \wedge b$ implies $x = a$ or $x = b$. We use $J(L)$ to denote the set of join-irreducible elements of $L$ and $M(L)$ to denote the set of meet-irreducible elements of $L$.
\end{definition}

\begin{theorem}[Fundamental Theorem of Formal Concept Analysis]
    Let $L$ be a finite lattice. The lattice $L$ can be recovered (up to isomorphism) from the data of the relation
    \[\{(x,y) \in J(L) \times M(L) : x \leq y\}\]
    from $J(L)$ to $M(L)$.
    
    The particular identity of the sets $J(L)$ and $M(L)$ is not important -- it is the isomorphism class of the above relation (as a morphism in the category $\mathbf{Rel}$ of sets and relations between them) which suffices to recover $L$ up to isomorphism. The isomorphism class of this relation is known as the \emph{reduced formal context} of $L$.
\end{theorem}

More generally, there is a construction $\mathfrak{B}$ which takes a relation between finite sets and produces a finite lattice, with the property that $\mathfrak{B}(R_1) \cong \mathfrak{B}(R_2)$ when $R_1$ and $R_2$ are isomorphic relations. A relation between finite sets $X$ and $Y$ can be viewed as a binary matrix with rows labeled by the elements of $Y$ and columns labeled by the elements of $X$. Two binary matrices produced in this way correspond to isomorphic relations if the rows and columns of one matrix can be permuted to obtain the other matrix. Thus, one can view $\mathfrak{B}$ as a construction which takes in a binary matrix (up to permutations of its rows and columns) and produces a finite lattice (well-defined up to isomorphism).

This construction is also insensitive to repeated rows and columns -- if a binary matrix $M$ contains a repeated row or column, one can remove it to obtain a different binary matrix $M'$ which nonetheless produces the same finite lattice. More generally, one can remove any row or column from $M$ which is the intersection\footnote{Here by an intersection of rows/columns we mean the product in a ring $(\mathbb{Z}/2\mathbb{Z})^k$.} of other rows or columns (respectively) without affecting the resulting finite lattice. A matrix $M$ such that no row or column is an intersection of other rows or columns (respectively) is said to be \emph{reduced}. The matrix corresponding to the relation $\{(x,y) \in J(L) \times M(L) : x \leq y\}$ appearing in the statement of the fundamental theorem of FCA is reduced, and the fundamental theorem of FCA has a converse, which says that two finite lattices are isomorphic if and only if the corresponding reduced matrices are equivalent via permuting rows and columns.

Thus, isomorphism classes of finite lattices are in bijective correspondence with equivalence classes of binary matrices; the binary matrix corresponding to a finite lattice is called its \emph{reduced formal context}. In \cite{BS}, Balchin and the second author describe the reduced formal context of the lattice $\Tr(L)$ for $L$ any finite lattice with $G$-action.

\begin{proposition}[{\cite[Summary 2.15]{BS}}]
    Let $L$ be a finite lattice with $G$-action. Then $J(\Tr(L))$ and $M(\Tr(L))$ are both in natural bijective correspondence with the set of $G$-orbits of nonidentity relations in $L$; precisely,
    \[J(\Tr(L)) = \{\lfloor x \to y \rfloor : x < y\}\]
    \[M(\Tr(L)) = \{\lceil x \to y \rceil^\boxslash : x < y\}.\]
    Moreover, the relation $\subseteq$ between the elements of $J(\Tr(L))$ and $M(\Tr(L))$ is given by
    \[
    \floor{a \to b} \subseteq \ceil{x \to y}^{\boxslash} \text{ if and only if, for all } g \in G,\, 
    g \cdot a \not\geq x \text{ or } g \cdot b \not\geq y \text{ or } g \cdot a \geq y.
    \]
\end{proposition}

In this paper, we are concerned with identifying the lattice $\Sat(\Sub(C_p^n))$, i.e. the lattice of saturated transfer systems on the subgroup lattice of an elementary abelian group. As a result, we would like to identify the sets $J(\Sat(\Sub(C_p^n)))$ and $M(\Sat(\Sub(C_p^n)))$, as well as the relation $\subseteq$ between them.

\section{Irreducible Saturated Transfer Systems}
In this section we identify the reduced formal context of $\Sat(L)$ for any finite modular lattice $L$ -- of primary interest is the case $L = \Sub(G)$ where $G$ is a finite abelian group.

\begin{lemma}\label{lem: saturated closure of transfer system}
    Let $L$ be a finite lattice. Let $\mathsf{T} \in \Tr(L)$, and let
    \[\mathsf{T}^\natural = \{b \to c : b \leq c, \exists a \leq b (a \xrightarrow{\mathsf{T}} c)\}.\]
    Then $\floor{\mathsf{T}}_\mathrm{sat} = \mathsf{T}^\natural$.
\end{lemma}
\begin{proof}
    It is clear by construction that
    \[\mathsf{T} \subseteq \mathsf{T}^\natural \subseteq \floor{\mathsf{T}}_\mathrm{sat}.\]
    Thus, it suffices to show that $\mathsf{T}^\natural$ is a saturated transfer system. By construction, $\mathsf{T}^\natural$ is a reflexive relation refining $\leq$. It remains to be shown that $\mathsf{T}^\natural$ is closed under restriction and 2-out-of-3.

    First, we check that $\mathsf{T}^\natural$ is closed under restriction. Let $b \xrightarrow{\mathsf{T}^\natural} c$ and $x \leq c$. By assumption, we have $a \xrightarrow{\mathsf{T}} c$ for some $a \leq b$. We can then form the diagram
    \[\begin{tikzcd}[ampersand replacement=\&]
    	x \& c \\
    	{x \wedge b} \& b \\
    	{x \wedge a} \& a
    	\arrow[from=1-1, to=1-2]
    	\arrow[from=2-1, to=1-1]
    	\arrow[from=2-1, to=2-2]
    	\arrow["{\mathsf{T}^\natural}"', from=2-2, to=1-2]
    	\arrow["{\mathsf{T}}", curve={height=-24pt}, from=3-1, to=1-1]
    	\arrow[from=3-1, to=2-1]
    	\arrow[from=3-1, to=3-2]
    	\arrow["{\mathsf{T}}"', curve={height=24pt}, from=3-2, to=1-2]
    	\arrow[from=3-2, to=2-2]
    \end{tikzcd}\]
    demonstrating that $x \wedge b \xrightarrow{\mathsf{T}^\natural} x$, as desired.

    Next, we check that $\mathsf{T}^\natural$ is closed under composition. Let $b \xrightarrow{\mathsf{T}^\natural} c \xrightarrow{\mathsf{T}^\natural} d$. In particular, we have $x \leq c$ such that $x \xrightarrow{\mathsf{T}} d$. Since we have already established that $\mathsf{T}^\natural$ is closed under restriction, we have $x \wedge b \xrightarrow{\mathsf{T}^\natural} x$. Thus, there is some $w \leq x \wedge b$ such that $w \xrightarrow{\mathsf{T}} x$, i.e.
    \[\begin{tikzcd}[ampersand replacement=\&]
    	w \\
    	{x \wedge b} \& x \\
    	b \& c \& d
    	\arrow[from=1-1, to=2-1]
    	\arrow["{\mathsf{T}}", from=1-1, to=2-2]
    	\arrow["{\mathsf{T}^\natural}"', from=2-1, to=2-2]
    	\arrow[from=2-1, to=3-1]
    	\arrow[from=2-2, to=3-2]
    	\arrow["{\mathsf{T}}", from=2-2, to=3-3]
    	\arrow["{\mathsf{T}^\natural}"', from=3-1, to=3-2]
    	\arrow["{\mathsf{T}^\natural}"', from=3-2, to=3-3]
    \end{tikzcd}\]
    In particular, we have $w \leq b$ and $w \xrightarrow{\mathsf{T}} d$, so $x \xrightarrow{\mathsf{T}^\natural} d$, as desired.

    Finally, we must show that $b \leq c \leq d$ and $b \xrightarrow{\mathsf{T}^\natural} d$ implies $c \xrightarrow{\mathsf{T}^\natural} d$. This is immediate: $b \xrightarrow{\mathsf{T}^\natural} d$ implies that there exists $a \leq b$ such that $a \xrightarrow{\mathsf{T}} d$, but then also $a \leq c$, so $c \xrightarrow{\mathsf{T}^\natural} d$.
\end{proof}

\begin{theorem}[cf. Theorem A]\label{thm: classification of J(Sat) and M(Sat)}
    Let $L$ be a finite modular lattice. The join- and meet-irreducible saturated transfer systems on $L$ are precisely the join- and meet-irreducible (respectively) transfer systems on $L$ which happen to also be saturated. These are precisely $\floor{x \to y}$ and $\ceil{\bot \to z}^\boxslash$ (respectively) as $x \to y$ ranges over the covering relations in $L$ and $z$ ranges over the non-minimum elements of $L$.
\end{theorem}

To prove this theorem, we will need control over the sets $\floor{x \to y}$ and $\ceil{x \to y}$. This is provided by a result of Rubin~\cite[Proposition A.2]{Rubin}.

\begin{proposition}
    Let $X$ be a set of arrows in a lattice $L$. Then
    \[\floor{X} = \{z \wedge a \to z : a \to b \in X, z \leq b\}^\circ,\]
    where ${(-)}^\circ$ denotes the reflexive-transitive closure of a binary relation.

    If $X$ is a singleton set, then one needs only take the reflexive closure.
\end{proposition}

\begin{lemma}\label{lem: saturated principal is covering relation}
    Let $L$ be a modular lattice, and let $x < y$ be a covering relation in $L$. Then $\floor{x \to y}$ consists only of identity and covering relations.
\end{lemma}
\begin{proof}
    Any non-identity arrow in $\floor{x \to y}$ is of the form $b \wedge x \to b$ for some $b \leq y$. Let $c \in L$ be such that $b \wedge x < c < b$. We have
    \[x \leq x \vee c = x \vee (c \wedge y) = (x \vee c) \wedge y \leq y,\]
    so $x \vee c = x$ or $x \vee c = y$. If $x \vee c = x$ then $c \leq x$, so $c \leq b \wedge x < c$, which is a contradiction. If $x \vee c = y$, then $c = c \vee (x \wedge b) = (c \vee x) \wedge b = y \wedge b = b$, which is again a contradiction.
\end{proof}

\begin{lemma}\label{lem: when principal floor is saturated}
    Let $L$ be a lattice, and let $x, y \in L$ with $x \leq y$. Then $\floor{x \to y}$ is saturated if and only if $x \to y$ is a cover relation.
\end{lemma}
\begin{proof}
    First, suppose $x \to y$ is not a cover relation, so that there is some $z \in L$ such that $x < z < y$. To show that $\floor{x \to y}$ is not saturated, it suffices to show that $z \to y \notin \floor{x \to y}$. Rubin's theorem shows that this is the case.

    In the other direction, suppose $x \to y$ is a cover relation. Let $a \to b \in \floor{x \to y}$ be arbitrary. By \Cref{lem: saturated principal is covering relation}, $a \to b$ is either an identity relation or a covering relation -- in either case, there is nothing to check to ensure that $\floor{x \to y}$ is saturated.
\end{proof}

\begin{lemma}
    Let $L$ be a lattice, and let $x,y \in L$ with $x \leq y$. Then
    \[\ceil{x \to y}^\boxslash = \{x \to y\}^\boxslash.\]
\end{lemma}
\begin{proof}
    Since $\{x \to y\} \subseteq \ceil{x \to y}$, we have $\ceil{x \to y}^\boxslash \subseteq \{x \to y\}^\boxslash$. In the other direction, let $a \to b \in \{x \to y\}^\boxslash$ be arbitrary. For any arrow $r \to s \in \ceil{x \to y}$, we have either $r = s$ (in which case $a \to b \in \{r \to s\}^\boxslash$) or $r \geq x$ and $s = y \vee r$. In the latter case, assume $r \leq a$ and $s \leq b$. We then conclude that $y \leq a$, and since $s = y \vee r$ we thus have $s \leq a$. So, $a \to b \in \{r \to s\}^\boxslash$.
\[\begin{tikzcd}[ampersand replacement=\&]
	x \& r \& a \\
	y \& s \& b
	\arrow[from=1-1, to=1-2]
	\arrow[from=1-1, to=2-1]
	\arrow[from=1-2, to=1-3]
	\arrow[from=1-3, to=2-3]
	\arrow[dashed, from=2-1, to=1-3]
	\arrow[from=2-1, to=2-2]
	\arrow[from=2-2, to=2-3]
\end{tikzcd}\]
Since $a \to b$ was arbitrary, we have $\{x \to y\}^\boxslash \subseteq \{r \to s\}^\boxslash$, and since $r \to s$ was arbitrary we have $\{x \to y\}^\boxslash \subseteq \ceil{x \to y}^\boxslash$, as desired.
\end{proof}

\begin{lemma}
    Let $L$ be a lattice, and let $x, y \in L$ with $x \leq y$. Then $\ceil{x \to y}^\boxslash$ is saturated if and only if $x = \bot$ or $x = y$.
\end{lemma}
\begin{proof}
    First, suppose $\ceil{x \to y}^\boxslash$ is saturated and $x \neq \bot$. Then
    \[\bot \to y \in \{x \to y\}^\boxslash = \ceil{x \to y}^\boxslash,\]
    so $x \to y \in \{x \to y\}^\boxslash$, which implies $x = y$.

    Next, if $x = y$, then $\ceil{x \to y}^\boxslash = L$ is a saturated transfer system.

    Finally, suppose $x = \bot$. Then
    \[\ceil{x \to y}^\boxslash = \{\bot \to y\}^\boxslash = \{a \to b : y \nleq b \text{ or } y \leq a\}.\]
    Let $a \leq b \leq c$ with $a \to c \in \ceil{x \to y}^\boxslash$. Then $y \nleq c$ or $y \leq a$. This implies $y \nleq c$ or $y \leq b$, i.e. $b \to c \in \ceil{x \to y}^\boxslash$.
\end{proof}

\begin{lemma}
    Let $L$ be a finite modular lattice. The join-irreducible saturated transfer systems on $L$ are precisely the transfer systems of the form $\floor{x \to y}$ for $x \to y \in \Cov(L)$.
\end{lemma}
\begin{proof}
    First, let $\mathsf{T} \in J(\Sat(L))$ be arbitrary. We know that $\mathsf{T} = \floor{\mathsf{T} \cap \Cov(L)} = \floor{\mathsf{T} \cap \Cov(L)}_\mathrm{sat}$, so
    \[\mathsf{T} = \bigvee_{c \to d \in \mathsf{T} \cap \Cov(L)}^\mathrm{sat} \floor{c \to d}_\mathrm{sat}\]
    where $\bigvee^\mathrm{sat}$ indicates that this join is taken in the lattice $\Sat(L)$. Since $\mathsf{T}$ is join-irreducible, we conclude that
    \[\mathsf{T} = \floor{c \to d}_\mathrm{sat}\]
    for some $c \to d \in \Cov(L)$. By \Cref{lem: when principal floor is saturated}, $\floor{c \to d}_\mathrm{sat} = \floor{c \to d}$, as desired.

    In the other direction, let $x \to y \in \Cov(L)$ be arbitrary (so that $\floor{x \to y}$ is saturated by \Cref{lem: when principal floor is saturated}), and suppose $\floor{x \to y} = \mathsf{T}_1 \vee^\mathrm{sat} \mathsf{T}_2$ for some $\mathsf{T}_1, \mathsf{T}_2 \in \Sat(L)$. In the notation of \Cref{lem: saturated closure of transfer system}, we have
    \[\mathsf{T}_1 \vee^\mathrm{sat} \mathsf{T}_2 = \floor{\mathsf{T}_1 \vee \mathsf{T}_2}_\mathrm{sat} = (\mathsf{T}_1 \vee \mathsf{T}_2)^\natural = \floor{\mathsf{T}_1 \cup \mathsf{T}_2}^\natural.\]
    We now employ \cite[Theorem A.2]{Rubin}, which tells us that\footnote{Rubin's theorem is stated specifically for subgroup lattices, but the same argument works for any finite lattice.}
    \[\floor{\mathsf{T}_1 \cup \mathsf{T}_2} = (\mathsf{T}_1 \cup \mathsf{T}_2)^\circ\]
    where $({-})^\circ$ denotes the reflexive-transitive closure of a binary relation. Now we have in particular that
    \[x \to y \in \left((\mathsf{T}_1 \cup \mathsf{T}_2)^\circ\right)^\natural,\]
    so there is some $w \leq x$ such that
    \[w \to y \in (\mathsf{T}_1 \cup \mathsf{T}_2)^\circ \subseteq \floor{x \to y}^\circ = \floor{x \to y}.\]
    Now, since $w \leq x < y$, we have $w \to y = z \wedge x \to z$ for some $z \leq y$, and thus $w = y \wedge x = x$. So far, we have established that
    \[x \to y \in (\mathsf{T}_1 \cup \mathsf{T}_2)^\circ.\]
    Thus, there is some $y \neq w' \geq x$ such that $w \to y \in T_1 \cup T_2$. Since $T_1 \cup T_2 \subseteq \floor{x \to y}$, we have $w = y \wedge x = x$. Thus, $x \to y \in \mathsf{T}_1 \cup \mathsf{T}_2$, so $\floor{x \to y} \subseteq \mathsf{T}_1$ or $\floor{x \to y} \subseteq \mathsf{T}_2$.
\end{proof}

\begin{lemma}
    Let $L$ be a finite modular lattice. The meet-irreducible saturated transfer systems on $L$ are precisely the transfer systems of the form $\ceil{\bot \to y}^\boxslash$ for $y \neq \bot$.    
\end{lemma}
\begin{proof}
    For $y \neq \bot$, $\ceil{\bot \to y}^\boxslash$ is a non-maximum element of $\Sat(L)$. If $\ceil{\bot \to y}^\boxslash = \mathsf{T}_1 \cap \mathsf{T}_2$, then $\ceil{\bot \to y}^\boxslash \in M(\Tr(L))$ implies that $\ceil{\bot \to y}^\boxslash = \mathsf{T}_1$ or $\ceil{\bot \to y}^\boxslash = \mathsf{T}_2$, as desired.

    By \cite[Corollary 4.2]{bose_combinatorics_2025}, if $\mathsf{T}$ is a meet-irreducible saturated transfer system on $L$, then
    \[{}^\boxslash \mathsf{T} = \lceil \{\bot \to x_i : i\} \rceil = \bigvee_i \lceil \bot \to x_i \rceil\]
    for some collection of objects $\{x_i\}_i$, where this join is performed in the lattice of cotransfer systems on $L$. Now
    \[\mathsf{T} = ({}^\boxslash \mathsf{T})^\boxslash = \left(\bigvee_i \lceil \bot \to x_i \rceil\right)^\boxslash = \bigwedge_i \lceil \bot \to x_i \rceil^\boxslash,\]
    where this meet is performed in the lattice of transfer systems on $L$ (recalling \Cref{prop:boxslash adjunction}). Since each $\lceil \bot \to x_i \rceil^\boxslash$ is saturated, this meet is equivalently performed in the lattice of saturated transfer systems on $L$. Now since $\mathsf{T}$ is meet-irreducible, we have
    \[\mathsf{T} = \lceil \bot \to y \rceil^\boxslash\]
    for some $y \in L$. If $y = \bot$ then $\mathsf{T} = L$, which is a contradiction.
\end{proof}

The above two lemmas, taken together, give precisely \Cref{thm: classification of J(Sat) and M(Sat)}. We now understand the elements of $J(\Sat(L))$ and $M(\Sat(L))$. Finally, we must determine the relation $\subseteq$ on $J(\Sat(L)) \times M(\Sat(L))$; since $J(\Sat(L)) \subseteq J(\Tr(L))$ and $M(\Sat(L)) \subseteq M(\Tr(L))$, this is immediate from the work in \cite{BS}.

\begin{lemma}[{\cite[Theorem 2.14]{BS}}]
    Let $L$ be a lattice, and let $a, b, y \in L$ be such that $a < b$ and $y \neq \bot$. Then
    \[\floor{a \to b} \subseteq \ceil{\bot \to y}^\boxslash\]
    if and only if
    \[y \nleq b \text{ or } y \leq a.\]
\end{lemma}

We now have in place a complete description of the reduced formal context of $\Sat(L)$ for any finite modular lattice $L$. \Cref{fig:C53context} shows the case $L = \Sub(C_5^3)$. Using the FCA software tool \texttt{PCbO} \cite{pcbo}, one obtains an enumeration of saturated transfer systems on the elementary abelian group $C_5^3$.

\begin{figure}[h]
    \centering
    \includegraphics[width=0.3\linewidth]{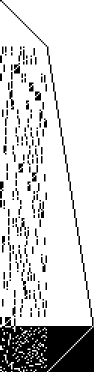}
    \caption{The reduced formal context for the lattice $\Sat(\Sub(C_5^3))$. A black pixel represents a $0$ in the binary matrix and a white pixel represents a $1$.}
    \label{fig:C53context}
\end{figure}

\begin{theorem}[cf. Theorem B]
    There are exactly $13,784,538,270,571$ saturated transfer systems on $C_5^3$.
\end{theorem}

\section{Density}\label{sec:density}

For $C_p^n$, there are

\begin{equation}
\label{eq:number-meet-irr-sat} \lvert M(\Sat(\Sub(C_p^n))) \rvert = \sum_{i=1}^n \qbinom{n}{i}{p}
\end{equation}

meet-irreducible saturated transfer systems, where $\qbinom{n}{i}{p}$ denotes the ``$p$-binomial coefficient''\footnote{These are also known as ``Gaussian binomial coefficients''. They are so named for the following reason: one can make sense of the expression defining $\qbinom{n}{i}{p}$ for any real number $p > 1$, and for fixed $n,i$ one has $\lim_{p \to 1^+} \qbinom{n}{i}{p} = \binom{n}{i}$.}
\[\qbinom{n}{i}{p} := \prod_{j=0}^{i-1} \frac{p^n - p^j}{p^i - p^j}.\]
$\qbinom{n}{i}{p}$ is the number of $i$-dimensional subspaces of the $\mathbb{F}_p$-vector space $\mathbb{F}_p^n$. The fact that the number of meet-irreducible saturated transfer systems is equal to the number of nontrivial subgroups of $C_p^n$ thus directly yields \eqref{eq:number-meet-irr-sat} above.

The total number of subspaces of $\mathbb{F}_p^n$, given by $\sum_{i=0}^n \qbinom{n}{i}{p}$, will be denoted $a_{n,p}$. We now record some useful facts about the quantities $\qbinom{n}{i}{p}$ and $a_{n,p}$.

\begin{lemma}\label{lem:number of subspaces recursion}
    For any natural number $n$ and any prime number $p$,
    \[a_{n+2,p} = 2a_{n+1,p} + (p^{n+1}-1)a_{n,p}.\]
\end{lemma}
The following proof is adapted from \cite{critzer_combinatorics_2018} -- we reproduce the proof here for completeness.
\begin{proof}
    Fix $n$ and $p$, and moreover fix a nonzero vector $v \in \mathbb{F}_p^{n+2}$. The number of subspaces of $\mathbb{F}_p^{n+2}$ which contain $v$ is equal to the number of subspaces of $\mathbb{F}_p^{n+2}/\operatorname{span}(v) \cong \mathbb{F}_p^{n+1}$, which by definition equals $a_{n+1,p}$.
    
    A $d$-dimensional subspace of $\mathbb{F}_p^{n+2}$ which does \emph{not} contain $v$ has a basis which is a subset of $\mathbb{F}_p^{n+2} \setminus \operatorname{span}(v)$. Thus, there are
    \[\frac{(p^{n+2}-p)\dots(p^{n+2}-p^d)}{(p^d-1)\dots(p^d-p^{d-1})} = p^d \qbinom{n+1}{d}{p}\]
    such subspaces.

    Now we know that the total number of subspaces of $\mathbb{F}_p^{n+2}$ which do not contain $v$ is
    \[\sum_{d=0}^{n+1} p^d \qbinom{n+1}{d}{p}.\]
    We note that this is also equal to the cardinality of the multiset
    \[\coprod_{W \leq \mathbb{F}_p^{n+1}} W,\]
    i.e. the total number of vectors appearing in all proper subspaces of $\mathbb{F}_p^{n+1}$, counted with multiplicity.

    Fix a nonzero vector $w \in \mathbb{F}_p^{n+1}$. By our prior work, the number of times $w$ appears in the above multiset is $a_{n,p}$. Thus, the total number of nonzero vectors in the above multiset is $(p^{n+1}-1)a_{n,p}$. Since the zero vector is an element of every subspaces of $\mathbb{F}_p^{n+1}$, the number of times the zero vector appears in the above multiset is $a_{n+1,p}$. Thus,
    \[\sum_{d=0}^{n+1} p^d \qbinom{n+1}{d}{p} = a_{n+1,p} + (p^{n+1}-1)a_{n,p},\]
    and we conclude that
    \[a_{n+2,p} = a_{n+1,p} + a_{n+1,p} + (p^{n+1}-1)a_{n,p},\]
    as desired.
\end{proof}

\begin{corollary}
    \label{cor:bounds on a}
    For all natural numbers $n$ and all prime numbers $p$,
    \[p^{(n^2-1)/4} \leq a_{n,p} \leq p^{(n+1.1)^2/4}.\]
\end{corollary}
\begin{proof}
    Fix a prime number $p$.
    
    We first establish the upper bound by induction on $n$. We establish the base cases for $n \leq 3$ directly:
    \[a_{0,p} = 1  = p^0 \leq p^{1.1^2/4}\]
    \[a_{1,p} = 2 \leq p \leq p^{1.1025} = p^{2.1^2/4}\]
    \[a_{2,p} = p+3 \leq p^{3.1^2/4}\]
    \[a_{3,p} = 2p^2+2p+4 \leq p^{4.1^2/4}\]
    where the last two inequalities can be established with basic calculus, using the fact that $p \geq 2$.
    For the inductive step, let $n \geq 2$. Then we have
    \begin{multline*}
        a_{n+2,p}
        = 2a_{n+1,p} + (p^{n+1}-1)a_{n,p}
        \leq 2p^{(n+2.1)^2/4} + (p^{n+1}-1)p^{(n+1.1)^2/4} \\
        = p^{(n+1.1)^2/4} \left( 2p^{n/2 + 1.6} + p^{n+1} - 1 \right)
        \leq p^{(n+1.1)^2/4} \left( 2p^{n+1} + p^{n+1} \right) \\
        = 3p^{(n+1.1)^2/4 + n + 1}
        \leq p^{(n+3.1)^2/4},
    \end{multline*}
    as desired.

    Next, we establish the lower bound, again by induction on $n$. We establish the bases cases for $n \leq 1$ directly:
    \[a_{0,p} = 1 \geq p^{-0.25} = p^{(0^2-1)/4}\]
    \[a_{1,p} = 2 \geq 1 = p^{(1^2-1)/4}\]
    For the inductive step, let $n \geq 0$. Then we have
    \begin{multline*}
        a_{n+2,p}
        = 2a_{n+1,p} + (p^{n+1}-1)a_{n,p}
        \geq 2p^{((n+1)^2-1)/4} + (p^{n+1}-1)p^{(n^2-1)/4} \\
        = p^{(n^2-1)/4} \left( 2p^{(2n+1)/4} + p^{n+1} - 1 \right)
        \geq p^{(n^2-1)/4} p^{n+1} = p^{((n+2)^2-1)/4},
    \end{multline*}
    as desired.
\end{proof}

\begin{proposition}
    The number of join-irreducible saturated transfer systems is
    \[\sum_{i=0}^{n-1} \qbinom{n}{i}{p}\qbinom{n-i}{1}{p} = \sum_{d=1}^n \qbinom{n}{d}{p} \qbinom{d}{1}{p} = \qbinom{n}{1}{p} a_{n-1,p}.\]
\end{proposition}
\begin{proof}
    To start we note certain properties of the ``$p$-binomial coefficient" \(\binom{n}{i}_p\). We have \(\binom{n}{1}_p = \frac{p^n - 1}{p - 1}\).
    By duality between $\mathbb{F}_p^n$ and its dual space, \[\binom{n}{k}_p = \binom{n}{n-k}_p.\]
    Setting $d = n-i$, \[\sum_{i=0}^{n-1} \qbinom{n}{i}{p}\qbinom{n-i}{1}{p}\]
    becomes
    \[\sum_{d=1}^{n} \qbinom{n}{n-d}{p}\qbinom{d}{1}{p},\]
    which can also be written as
    \[\sum_{d=1}^{n} \qbinom{n}{d}{p}\qbinom{d}{1}{p}.\]

    Next, we wish to show that $\sum_{d=1}^n \qbinom{n}{d}{p} \qbinom{d}{1}{p} = \qbinom{n}{1}{p} a_{n-1,p}$. The left-hand side is also the number of pairs $(A,B)$ of subspaces of $\mathbb{F}_p^n$ such that $\dim A = 1$ and $A \leq B$. The right-hand size is the number of pairs $(A,C)$ where $A$ is a $1$-dimensional subspace of $\mathbb{F}_p^n$ and $C$ is a subspace of $\mathbb{F}_p^n/A$. By the correspondence theorem, these numbers agree.
\end{proof}

We have a $0$ in the reduced formal context for saturated transfer systems on $\operatorname{Sub}(C_p^n)$ for each triple $(A,B,K)$ of subgroups of $C_p^n$ where:
\begin{itemize}
    \item $A \leq B$ and $[B : A] = p$ (i.e. $A \to B$ is a direct edge),
    \item $B \geq K$ and $A \ngeq K$.
\end{itemize}

We therefore obtain the following count of the number of $0$'s in the reduced context:

\[\sum_{d=1}^n \qbinom{n}{d}{p} \qbinom{d}{d-1}{p} (a_{d,p} - a_{d-1,p}) = \sum_{d=1}^n \qbinom{n}{d}{p} \qbinom{d}{1}{p} (a_{d,p} - a_{d-1,p})\]

and thus the total density is

\[\delta(n,p) := 1 - \frac{\sum_{d=1}^n \qbinom{n}{d}{p} \qbinom{d}{1}{p} (a_{d,p} - a_{d-1,p})}{\qbinom{n}{1}{p}(a_{n,p}-1)a_{n-1,p}}\]

\begin{proposition}
    For every prime number $p$, $\delta(2,p) = 1/2$.
\end{proposition}
\begin{proof}
    To start, we substitute $n=2$ in the formula for total density to get
     \[\delta(2,p) := 1 - \frac{\sum_{d=1}^2 \qbinom{2}{d}{p} \qbinom{d}{1}{2} (a_{d,p} - a_{d-1,p})}{\qbinom{n}{1}{2}(a_{2,p}-1)a_{2-1,p}}.\]
     This becomes
     \[\delta(2,p) := 1 - \frac{ \qbinom{2}{1}{p} \qbinom{1}{1}{p} (a_{1,p} - a_{0,p})+\qbinom{2}{2}{p} \qbinom{2}{1}{p} (a_{2,p} - a_{1,p})}{\qbinom{n}{1}{2}(a_{2,p}-1)a_{2-1,p}}.\]
     By direct computation, this is
    \[\delta(2,p) := 1 - \frac{ (1+p)(1)(2-1)+(1)(1+p)(3+p-2)}{(1+p)(3+p-1)(2)}\]
    which simplifies to
    \[\delta(2,p) := 1 - \frac{ (1+p)+(1+p)^2}{(1+p)(2+p)(2)}=1-\frac{2+p}{(2+p)(2)}=1/2.\qedhere\]
\end{proof}

\begin{lemma}\label{lem:power inequality}
Let $m$ and $n$ be real numbers with $m \geq 2$ and $n \geq 3$. Then
\[m^{(n^2-1)/4} - 1 \geq m^{(n^2-3)/4}.\]
\end{lemma}
\begin{proof}
    We consider the quantity $f(m,n) := m^{(n^2-1)/4} - 1 - m^{(n^2-3)/4}$ as a smooth function of $m$ and $n$. The desired inequality equivalently states that $f(m,n) \geq 0$ when $m \geq 2$ and $n \geq 3$. It is easy to check that $f(2,3) = 3-\sqrt{2} > 0$, so it suffices to show that $\frac{\partial f}{\partial m}$ and $\frac{\partial f}{\partial n}$ are positive when $m > 2$ and $n > 3$.
    
    We compute
    \[\frac{\partial f}{\partial m} = \frac{1}{4} m^{\frac{1}{4} \left(n^2-7\right)} \left(\sqrt{m} \left(n^2-1\right)-n^2+3\right)\]
    and
    \[\frac{\partial f}{\partial n} = \frac{1}{2} \left(\sqrt{m}-1\right) n m^{\frac{1}{4} \left(n^2-3\right)} \log (m).\]
    Since $m \geq 2$ and $n > 1$, we have
    \[
        \sqrt{m} (n^2-1) - n^2 + 3 \geq \sqrt{2}(n^2 - 1) - n^2 + 3 = (\sqrt{2}-1)n^2 + (3 - \sqrt{2}) > 0.
    \]
    Since $m > 0$, this yields
    \[\frac{\partial f}{\partial m} = \frac{1}{4} m^{\frac{1}{4} \left(n^2-7\right)} \left(\sqrt{m} \left(n^2-1\right)-n^2+3\right) > 0.\]
    Next, since $m > 1$ and $n > 0$ we have
    \[\frac{\partial f}{\partial n} = \frac{1}{2} \left(\sqrt{m}-1\right) n m^{\frac{1}{4} \left(n^2-3\right)} \log (m) > 0.\qedhere\]
\end{proof}

\begin{proposition}
    Let $n > 2$ be fixed. We have
    \[\lim_{p \to \infty} \delta(n,p) = 1.\]
\end{proposition}
\begin{proof}
   We will show that $1-\delta$ approaches $0$ for a fixed $n$ as $p$ approaches $\infty$. Since $p$ is prime and $n \geq 3$, we have
    \begin{multline*}
        1 - \delta(n,p)
        = \frac{\sum_{d=1}^n \qbinom{n}{d}{p} \qbinom{d}{1}{p} (a_{d,p} - a_{d-1,p})}{\qbinom{n}{1}{p}(a_{n,p}-1)a_{n-1,p}}
        \leq \frac{\sum_{d=1}^n \qbinom{n}{d}{p} \qbinom{d}{1}{p} a_{d,p}}{\qbinom{n}{1}{p}(a_{n,p}-1)a_{n-1,p}} \\
        \leq \frac{\sum_{d=1}^n \qbinom{n}{d}{p} p^d p^{(d+1.1)^2/4}}{p^{n-1} (p^{(n^2-1)/4}-1) p^{((n-1)^2-1)/4}} \\
        \leq \frac{\sum_{d=1}^n \qbinom{n}{d}{p} p^d p^{(d+1.1)^2/4}}{p^{n-1} p^{(n^2-3)/4} p^{((n-1)^2-1)/4}}\\
        = p^{-\frac{1}{4} \left(2 n^2+2 n-7\right)} \sum_{d=1}^n \qbinom{n}{d}{p} p^{d + (d+1.1)^2/4},
    \end{multline*}
    using \Cref{cor:bounds on a}, \Cref{lem:power inequality}, and the estimates
    \[p^{n-1} \leq \underbrace{p^{n-1} + p^{n-2} + \dots + 1}_{=\qbinom{n}{1}{p}} \leq p^n.\]
    By \cite[Lemma 2.1]{IhringerThesis}, we also have
    \[\qbinom{n}{d}{p} < \frac{111}{32} p^{d(n-d)},\]
    so
    \[1 - \delta(n,p) \leq \frac{111}{32} p^{-\frac{1}{4} \left(2 n^2+2 n-7\right)} \sum_{d=1}^n p^{d(n-d) + d + (d+1.1)^2/4}.\]
    For fixed $n$, the expression $d(n-d) + d + (d+1.1)^2/4$ is maximized at $d = 31/30 + 2n/3$. Thus,
    \[1 - \delta(n,p) \leq \frac{111}{32} p^{-\frac{1}{4} \left(2 n^2+2 n-7\right)} \sum_{d=1}^n p^{31/30 + 2n/3} = \frac{111n}{32} p^{-\frac{1}{4} \left(2 n^2+2 n-7\right) + 31/30 + 2n/3}.\]
    The exponent $-\frac{1}{4} \left(2 n^2+2 n-7\right) + 31/30 + 2n/3 = 167/60 + n/6 - n^2/2$ is decreasing in $n$ for $n > 1/6$, and equals $-73/60$ for $n = 3$, so we obtain
    \[p^{-\frac{1}{4} \left(2 n^2+2 n-7\right) + 31/30 + 2n/3} = p^{167/60 + n/6 - n^2/2} \leq p^{-73/60} \leq p^{-1}.\]
    Now
    \[1 - \delta(n,p) \leq \frac{111n}{32} p^{-1}\]
    approaches $0$ as $p \to \infty$.
\end{proof}

\begin{proposition}
    For fixed $p$, we have
    \[\lim_{n \to \infty} \delta(n,p) = 1.\]
\end{proposition}
\begin{proof}
    We will equivalently show that $1-\delta$ approaches $0$. We have
    \begin{multline*}
        1 - \delta(n,p)
        = \frac{\sum_{d=1}^n \qbinom{n}{d}{p} \qbinom{d}{1}{p} (a_{d,p} - a_{d-1,p})}{\qbinom{n}{1}{p}(a_{n,p}-1)a_{n-1,p}}
        \leq \frac{\sum_{d=1}^n \qbinom{n}{d}{p} \qbinom{d}{1}{p} a_{d,p}}{\qbinom{n}{1}{p}(a_{n,p}-1)a_{n-1,p}} \\
        \leq \frac{\sum_{d=1}^n \qbinom{n}{d}{p} p^d p^{(d+1.1)^2/4}}{p^{n-1} (p^{(n^2-1)/4}-1) p^{((n-1)^2-1)/4}}
    \end{multline*}
    using the upper and lower bounds from \cref{cor:bounds on a} together with the estimates
    \[p^{n-1} \leq \underbrace{p^{n-1} + p^{n-2} + \dots + 1}_{=\qbinom{n}{1}{p}} \leq p^n.\]
    For sufficiently large $n$, we have $p^{(n^2-1)/4}-1 \geq p^{(n^2-2)/4}$, so we eventually have
    \[1 - \delta(n,p) \leq \frac{\sum_{d=1}^n \qbinom{n}{d}{p} p^{(d^2+6.2d+1.21)/4}}{p^{(n^2+n-3)/2}}.\]
    By \cite[Lemma 2.1]{IhringerThesis}, we also have
    \[\qbinom{n}{d}{p} < \frac{111}{32} p^{d(n-d)},\]
    so (for sufficiently large $n$),
    \[1 - \delta(n,p) \leq \frac{111}{32} p^{1.8025} \frac{\sum_{d=1}^n p^{nd-0.75d^2+1.55d}}{p^{(n^2+n)/2}}.\]
    For fixed $n$, the expression $nd - 0.75d^2 + 1.55d$ is maximized at $d = 2n/3 + 31/30$. Thus, we obtain (for $n \gg 0$)
    \begin{multline*}
        1 - \delta(n,p) 
        \leq \frac{111}{32} p^{1.8025} \frac{n p^{n(31/30+2n/3)-0.75(31/30+2n/3)^2+1.55(31/30+2n/3)}}{p^{(n^2+n)/2}} \\
        = \frac{111}{32} n p^{1.8025} \frac{p^{n^2/3+31n/30+961/1200}}{p^{(n^2+n)/2}} \\
        = \frac{111}{32} n p^{-n^2/6+8n/15+781/300}
    \end{multline*}
    which approaches $0$ as $n \to \infty$.
\end{proof}

\printbibliography
\end{document}